\documentclass[11pt,reqno]{amsart}
\usepackage{amscd,amssymb,amsmath,amsthm}
\usepackage{mathtools}
\usepackage[english]{babel}
\usepackage[T1]{fontenc}
\usepackage[arrow,matrix]{xy}
\usepackage{graphicx}
\usepackage{tikz}
\usetikzlibrary{positioning,matrix,arrows,calc}
\usepackage{cite}
\usepackage{dsfont}
\newtheorem{theorem}[subsection]{Theorem}
\newtheorem{lemma}[subsection]{Lemma}

\theoremstyle{definition}

\numberwithin{equation}{section} 

\newcommand{\F}{{\mathcal F}}

\newcommand{\A}{{\mathcal A}}

\newcommand{\dd}{{\mathrm d}}

\newcommand{\bea}{\begin{eqnarray}}
\newcommand{\eea}{\end{eqnarray}}

\newcommand{\N}{\mathbb{N}}
\newcommand{\Z}{\mathbb{Z}}

\newcommand{\R}{\mathbb{R}}

\newcommand{\EE}{\mathbb{E}}
\newcommand{\PP}{\mathbb{P}}

\def \> {\Rightarrow}
\def \0 {\emptyset}

\def\eins{{\mathchoice {1\mskip-4mu\mathrm l}
{1\mskip-4mu\mathrm l}
{1\mskip-4.5mu\mathrm l} {1\mskip-5mu\mathrm l}}}

\begin{document}
\title[Asymptotics for a class of iterated random cubic operators]{Asymptotics for a class of iterated random cubic operators}

\author{A.J. Homburg, U.U. Jamilov, M. Scheutzow}

 \address{A.\ J.\ Homburg\\ KdV Institute for Mathematics, University of Amsterdam, Science park 107, 1098 XG Amsterdam, Netherlands\newline Department of Mathematics, VU University Amsterdam, De Boelelaan 1081, 1081 HV Amsterdam, Netherlands}
 \email{a.j.homburg@uva.nl}

  \address{U.\ U.\ Jamilov\\ V.I. Romanovskiy Institute of Mathematics, Academy of Sciences of Uzbekistan,
81, Mirzo Ulugbek str., 100170, Tashkent, Uzbekistan.}
\email{jamilovu@yandex.ru}

\address{M.\ Scheutzow \\ Institut f\"ur Mathematik, MA 7-5, Fakult\"at II,
        Technische Universit\"at Berlin, Stra\ss e des 17.~Juni 136, 10623 Berlin, FRG;}
\email {ms@math.tu-berlin.de}

\begin{abstract}
	We consider a class of cubic stochastic operators that are motivated by models
	for evolution of frequencies of genetic types in populations.
	We take populations with three mutually exclusive genetic types.

The long term dynamics of single maps, starting with a generic initial condition where
in particular all genetic types occur with positive frequency,
is asymptotic to equilibria where either only one genetic type survives, or
where all three genetic types occur.

We consider a family of independent and identically distributed maps from this class
and study its long term dynamics, in particular its random point attractors.
The long term dynamics of the random composition of maps is asymptotic, almost surely, to equilibria.
In contrast to the deterministic system, for generic initial conditions these can be equilibria with one or two or three types present
(depending only on the distribution).
\end{abstract}

\subjclass[2010] {37N25,37H10}

\keywords{random Volterra operators, random point attractors}

\maketitle

\section{Introduction}

%

Quadratic stochastic operators (QSO's) and their application in biology were first considered by Bernstein \cite{Br}.
They arise as models of genetic evolution that describe the dynamics of gene frequencies in mathematical population genetics.
These and similar models are also considered in their own right from a dynamical systems perspective.

Consider a population of $m$ different genetic types, where individuals possess a single type.
The frequency of occurrence of the different types is given by a probability distribution $x = (x_1,\ldots,x_m)$ and thus by a point
$x$ on the simplex $\Delta^{m-1}$.
Let $p_{ij,k}$ be the conditional probability that two individuals of type $i$ and $j$ given they interbreed, produce offspring of type $k$.
The expected gene distribution in the next generation is modeled by a
quadratic operator $V: \Delta^{m-1} \to \Delta^{m-1}$,
\[
(V(x))_k = \sum_{i,j=1}^m p_{ij,k} x_i x_j.
\]
The quadratic operator is a Volterra operator if $p_{ij,k} =  0$ for any $k \not\in \{i,j\}$. That is, the genetic type of an offspring
is a copy of one of its parents.

Several publications treat dynamics of Volterra quadratic operators.
A randomization is considered in \cite{JKM17}. That paper studies random dynamical systems obtained from
sequences of independent and identically distributed quadratic operators.
Under some conditions it is shown that trajectories converge to one of the vertices of $\Delta^{m-1}$ almost surely.

A generalization is the concept of
cubic stochastic operators (CSO).
A CSO is given by a map $V: \Delta^{m-1} \to \Delta^{m-1}$ of the form
\[
(V(x))_l = \sum_{i,j,k=1}^m p_{ijk,l} x_i x_j x_k
\]
with $p_{ijk,l} \ge 0$ and $\sum_{l=1}^m  p_{ijk,l} = 1$.
A CSO is called a Volterra operator if $p_{ijk,l}=0$ if $l \not\in \{i,j,k\}$.
We will consider a family of Volterra CSO's with coefficients of heredity
\[
p_{ijk,l} = \left\{  \begin{array}{ll}  1, &  \delta_{il} + \delta_{jl} + \delta_{kl} = 3, \\
                                         \theta, &  \delta_{il} + \delta_{jl} + \delta_{kl} = 2, \delta_{ij} + \delta_{ik} + \delta_{jk} =1, \\
                                         1-\theta, & \delta_{il} + \delta_{jl} + \delta_{kl} = 1, \delta_{ij} + \delta_{ik} + \delta_{jk} =1, \\
                                         \frac{1}{3}, & \delta_{il} + \delta_{jl} + \delta_{kl} = 1, \delta_{ij} + \delta_{ik} + \delta_{jk} =0, \\
                                         0, & \textrm{otherwise},
   \end{array}   \right.
\]
where $\delta_{ij}$ is the Kronecker symbol and $\theta \in [0,1]$.
Properties of such CSO's are studied in \cite{JKL17}.
%


We will take the approach of \cite{JKM17} and consider random iterations of CSO's.
For this we restrict to populations with three types, that is, to maps on the two dimensional simplex $\Delta^2$.
The resulting map $V_\theta:\Delta^2 \to \Delta^2$ is given by
\begin{equation}\label{cneop}
\left\{\begin{array}{l}
   (V_\theta(x))_1=x_1\big( x_1^2+3\theta x_1 (x_2+x_3)+3(1-\theta)(x_2^2 + x_3^2) + 2x_2 x_3\big), \\[2mm]
   (V_\theta(x))_2 =x_2\big(x_2^2+3\theta x_2 (x_3+x_1)+3(1-\theta)(x_3^2 + x_1^2) + 2x_3 x_1\big), \\[2mm]
   (V_\theta(x))_3=x_3\big( x_3^2+3\theta x_3 (x_1+x_2)+3(1-\theta)(x_1^2 + x_2^2) + 2x_1 x_2\big).
 \end{array}\right.
\end{equation}
The system can be seen as a discrete version of Kolmogorov equations for three interacting species.
It is easy to see that $V_{2/3}$ is the identity map.
It is shown in \cite{JKL17} that orbits of $V_\theta$, $\theta \ne \frac{2}{3}$, converge to one of the equilibria for which
none, one or two of the genetic types go extinct.

The simplex $\Delta^2$ consists of different invariant regions as depicted in Figure~\ref{f:parti}.
Write
\[
G_1=\{y \in \Delta^2: y_1 \ge y_2 \ge y_3\};
\]
$G_1$ is a subset of $\Delta^2$ which is invariant under each $V_\theta$.
The simplex $\Delta^2$ is the union of $G_1$ and its symmetric images under permutations of coordinates, see Figure~\ref{f:parti}.
\begin{figure}[!ht]
	\begin{tikzpicture}
	\draw (0,0)--(4,0)--($(2,{2*sqrt(3)})$)--(0,0);
	\draw (0,0)--($(3,{sqrt(3)})$);
	\draw ($(2,{2*sqrt(3)})$)--(2,0);
	\draw ($(1,{sqrt(3)})$)--(4,0);
	
	\draw (0,0) node[below]{$\mathbf{e}_1$};
	\draw (4,0) node[below]{$\mathbf{e}_2$};
	\draw ($(2,{2*sqrt(3)})$) node[above]{$\mathbf{e}_3$};
	
	\draw (1.8,0.4) node[left]{$G_1$};
	\draw (1.4,1.1) node[left]{$G_2$};
	\draw (1.9,2) node[left]{$G_3$};
	\draw (2.1,2) node[right]{$G_4$};
	\draw (2.5,1.1) node[right]{$G_5$};
	\draw (2.3,.4) node[right]{$G_6$};
	
	\end{tikzpicture}\hspace{2cm}
	\begin{tikzpicture}
	\draw (0,0)--(4,0)--($(4,{4*sqrt(3)/3})$)--(0,0);

	\draw (0,0) node[below]{$\mathbf{e}_1$};
	\draw (4,0) node[below]{$c$};
	\draw ($(4,{4*sqrt(3)/3})$) node[above]{$C$};
	
	\draw (3.2,0.8) node[left]{$G_1$};
	\draw (2,0) node[below]{$\Gamma_{12}$};
	\draw ($(4,{2*sqrt(3)/3})$) node[right]{$M_{12}$};
	\draw ($(2,{2.2*sqrt(3)/3})$) node[above]{$M_{23}$};
	
	\end{tikzpicture}
	\caption{Left: the partition of $\Delta^2$ in invariant triangles. Right: notation for $G_1$.}\label{f:parti}
\end{figure}
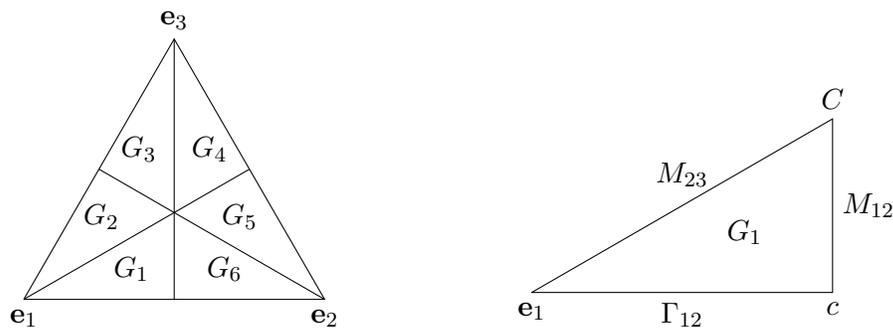
The boundary of $G_1$ is the union of three edges,
$M_{12} = \{y \in \Delta^2: y_1 = y_2 \ge y_3, \}$,  $M_{23} = \{y \in \Delta^2: y_1 \ge y_2 = y_3, \}$, $\Gamma_{12} = \{y \in \Delta^2: y_1 \ge y_2,  y_3=0 \}$.
Let $C =(1/3,1/3,1/3)$
be the center of $\Delta^2$, $c=(1/2,1/2,0)$,
$\mathbf{e}_1=(1,0,0)$ be the two other vertices of $G_1$.
These vertices are the fixed points of $V_\theta$ restricted to $G_1$, for any $\theta \ne \frac{2}{3}$ \cite{JKL17}.
The set of fixed points of $V_\theta$ on $\Delta^2$ consists of these points and their symmetric images, in total $7$ points.
Let $\A$ denote this set of fixed points.
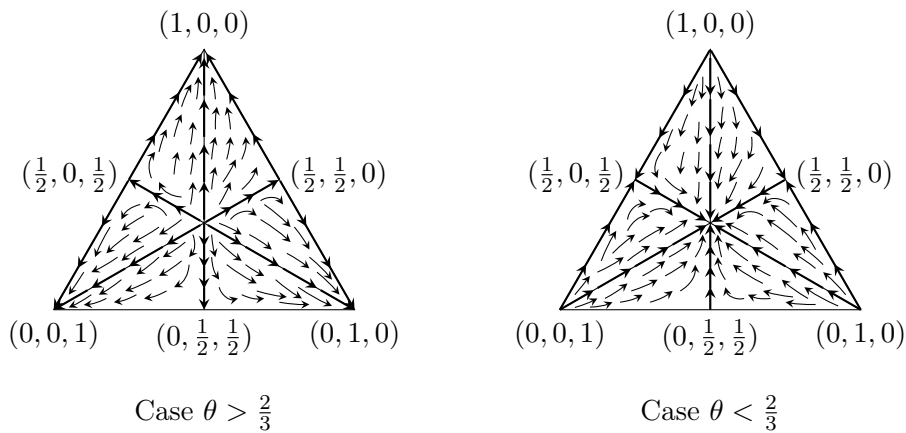
\begin{figure}[h!]
	\begin{tikzpicture}
	\draw (0,0)--(4,0)--($(2,{2*sqrt(3)})$)--(0,0);
	\draw (0,0)--($(3,{sqrt(3)})$);
	\draw ($(2,{2*sqrt(3)})$)--(2,0);
	\draw ($(1,{sqrt(3)})$)--(4,0);
	
	\foreach \x in {1, 1.32, 1.62}
	\draw[-stealth,thick,black]  ($(\x, {sqrt(3)*\x})$) -- ($(\x+ 0.35,{sqrt(3)*\x+ sqrt(3)*0.35})$); 
	\foreach \x in {0, 0.32, 0.64}
	\draw[stealth-,thick,black]  ($(\x, {sqrt(3)*\x})$) -- ($(\x+ 0.35,{sqrt(3)*\x+ sqrt(3)*0.35})$);
	
	\foreach \x in {3, 2.68, 2.38}
	\draw[-stealth,thick,black]  ($(\x, {-sqrt(3)*\x+4*sqrt(3)})$) -- ($(\x- 0.35,{-sqrt(3)*\x+ 4*sqrt(3)+sqrt(3)*0.35})$);  
	\foreach \x in {4, 3.68, 3.36}
	\draw[stealth-,thick,black] ($(\x, {-sqrt(3)*\x+4*sqrt(3)})$) -- ($(\x- 0.35,{-sqrt(3)*\x+ 4*sqrt(3)+sqrt(3)*0.35})$);
	
	\foreach \y in {1.2, 1.7, 2.3, 2.85}
	\draw[-stealth,thick,black]  (2, \y) -- (2,\y+0.5); 
	\foreach \y in {0, .28, .56,.84}
	\draw[stealth-,thick,black]  (2, \y) -- (2,\y+0.3);
	
	\foreach \x in {2.02, 2.32, 2.64}
	\draw[-stealth,thick,black]($(\x, {1/3*sqrt(3)*\x})$) -- ($(\x+ 0.35,{1/3*sqrt(3)*\x+ 1/3*sqrt(3)*0.35})$); 
	\foreach \x in {0, 0.5, 1,1.5}
	\draw[stealth-,thick,black]  ($(\x, {1/3*sqrt(3)*\x})$) -- ($(\x+ 0.45,{1/3*sqrt(3)*\x+ 1/3*sqrt(3)*0.45})$);

	\foreach \x in {2.5,3,3.5,4}
	\draw[stealth-,thick,black] ($(\x, {-1/3*sqrt(3)*\x+4/3*sqrt(3)})$) -- ($(\x- 0.45,{-1/3*sqrt(3)*\x+ 4/3*sqrt(3)+1/3*sqrt(3)*0.45})$);   
	\foreach \x in {1.95,1.65,1.35}
	\draw[-stealth,thick,black]  ($(\x, {-1/3*sqrt(3)*\x+4/3*sqrt(3)})$) -- ($(\x- 0.35,{-1/3*sqrt(3)*\x+ 4/3*sqrt(3)+1/3*sqrt(3)*0.35})$);
	
	\draw (0,0) node[below]{$(0,0,1)$};
	\draw (2,0) node[below]{$(0,\frac{1}{2},\frac{1}{2})$};
	\draw (4,0) node[below]{$(0,1,0)$};
	\draw (3,1.8) node[right]{$(\frac{1}{2},\frac{1}{2},0)$};
	\draw ($(2,{2*sqrt(3)})$) node[above]{$(1,0,0)$};
	\draw (1,1.8) node[left]{$(\frac{1}{2},0,\frac{1}{2})$};
	\draw (2,-1) node[below]{Case $\theta >\frac{2}{3}$};

	\path[-stealth,black] (2.1,1.6) edge[bend left=5] node { } (2.15,1.9);
	\path[-stealth,black] (2.5,1.6) edge[bend right=15] node { } (2.7,1.9);
	\path[-stealth,black] (2.5,2) edge[bend right=15] node { } (2.55,2.3);
	\path[-stealth,black] (2.1,2) edge[bend left=5] node { } (2.15,2.3);
	\path[-stealth,black] (2.3,2) edge[bend right=5] node { } (2.35,2.3);
	\path[-stealth,black] (2.35,2.4) edge[bend right=5] node { } (2.3,2.7);
	\path[-stealth,black] (2.15,2.4) edge[bend left=5] node { } (2.17,2.7);
	\path[-stealth,black] (2.15,2.8) edge[bend right=5] node { } (2.1,3.1);
	\path[-stealth,black] (2.2,1.4) edge[bend right=5] node { } (2.35,1.7);
	\path[-stealth,black] (2.25,1.7) edge[bend right=5] node { } (2.42,2);
	
	\path[-stealth,black] (1.82,1.8) edge[bend right=5] node { } (1.85,2.1);
	\path[-stealth,black] (1.82,2.2) edge[bend right=5] node { } (1.85,2.5);
	\path[-stealth,black] (1.6,2.3) edge[bend left=5] node { } (1.7,2.6);
	\path[-stealth,black] (1.7,2.7) edge[bend left=5] node { } (1.83,3);
	\path[-stealth,black] (1.9,2.8) edge[bend right=5] node { } (1.92,3.1);
	\path[-stealth,black] (1.4,2.1) edge[bend left=5] node { } (1.5,2.4);
	\path[-stealth,black] (1.7,1.5) edge[bend left=30] node { } (1.4,1.7);
	\path[-stealth,black] (1.3,1.75) edge[bend left=30] node { } (1.35,2.1);
	\path[-stealth,black] (1.6,1.8) edge[bend left=5] node { } (1.65,2.1);
	\path[-stealth,black]  (1.82,1.4) edge[bend right=5] node { } (1.85,1.7);
	
	\path[-stealth,black] (1.85,1) edge[bend left=30] node { } (1.8,.7);
	\path[-stealth,black] (1.75,.9) edge[bend left=5] node { } (1.45,.7);
	\path[-stealth,black] (1.4,.7) edge[bend left=5] node { } (1.1,.47);
	\path[-stealth,black] (.9,.4) edge[bend left=5] node { } (.6,.2);
	\path[-stealth,black] (.5,.2) edge[bend left=5] node { } (.2,.03);
	\path[-stealth,black] (1.8,.6) edge[bend left=30] node { } (1.6,.3);
	\path[-stealth,black] (1.5,.25) edge[bend left=10] node { } (1.2,.15);
	\path[-stealth,black] (1.1,.2) edge[bend left=10] node { } (0.8,.1);
	\path[-stealth,black] (.7,.15) edge[bend left=10] node { } (0.4,.065);
	\path[-stealth,black] (1.3,.4) edge[bend left=5] node { } (.95,.25);
	\path[-stealth,black] (1.7,.65) edge[bend left=5] node { } (1.4,.45);
	
	\path[-stealth,black] (.85,.7) edge[bend left=5] node { } (.5,.5);
	\path[-stealth,black] (.75,1.1) edge[bend right=5] node { } (.55,.7);
	\path[-stealth,black] (.4,.5) edge[bend left=1] node { } (.2,.2);
	\path[-stealth,black] (1.7,1.2) edge[bend left=5] node { } (1.4,1);
	\path[-stealth,black] (1.2,1.25) edge[bend right=5] node { } (.9,1.05);
	\path[-stealth,black] (1.2,.85) edge[bend right=5] node { } (.9,.65);
	\path[-stealth,black] (1,1) edge[bend right=5] node { } (.7,0.75);
	\path[-stealth,black] (1.4,1.1) edge[bend left=5] node { } (1.05,0.9);
	\path[-stealth,black] (1.5,1.25) edge[bend right=45] node { } (1.2,1.35);
	\path[-stealth,black] (1.15,1.4) edge[bend right=25] node { } (.9,1.2);
	
	\path[-stealth,black] (2.1,1) edge[bend right=5] node { } (2.15,.7);
	\path[-stealth,black] (2.1,.7) edge[bend right=15] node { } (2.2,.4);
	\path[-stealth,black] (2.2,.3) edge[bend right=45] node { } (2.45,.15);
	\path[-stealth,black] (2.5,.2) edge[bend right=15] node { } (2.8,.15);
	\path[-stealth,black] (3,.1) edge[bend right=5] node { } (3.4,.07);
	\path[-stealth,black] (2.2,.9) edge[bend right=5] node { } (2.5,.65);
	\path[-stealth,black] (2.55,.6) edge[bend right=5] node { } (2.8,.45);
	\path[-stealth,black] (2.95,.4) edge[bend right=5] node { } (3.3,.15);
	\path[-stealth,black] (3.25,.3) edge[bend right=5] node { } (3.7,.05);
	\path[-stealth,black] (2.3,.6) edge[bend right=5] node { } (2.7,.35);
	\path[-stealth,black] (2.7,.25) edge[bend right=5] node { } (3.1,.17);
	
	\path[-stealth,black] (3.5,.5) edge[bend left=5] node { } (3.8,.2);
	\path[-stealth,black] (3.1,.75) edge[bend left=5] node { } (3.4,.55);
	\path[-stealth,black] (2.6,1) edge[bend left=5] node { } (2.9,.8);
	\path[-stealth,black] (2.2,1.15) edge[bend left=5] node { } (2.5,1.0);
	\path[-stealth,black] (2.75,1.2) edge[bend left=5] node { } (3.15,.9);
	\path[-stealth,black] (2.8,1.4) edge[bend left=5] node { } (3.15,1.1);
	\path[-stealth,black] (2.35,1.2) edge[bend left=45] node { } (2.7,1.25);
	\path[-stealth,black] (2.75,1.45) edge[bend left=45] node { } (3.05,1.4);
	\path[-stealth,black] (3.2,1.1) edge[bend left=5] node { } (3.4,.8);
	\path[-stealth,black] (3.4,.7) edge[bend left=5] node { } (3.7,.4);
	\end{tikzpicture} \hspace{1cm}
	\begin{tikzpicture}
	\draw (0,0)--(4,0)--($(2,{2*sqrt(3)})$)--(0,0);
	\draw (0,0)--($(3,{sqrt(3)})$);
	\draw ($(2,{2*sqrt(3)})$)--(2,0);
	\draw ($(1,{sqrt(3)})$)--(4,0);
	
	\foreach \x in {1, 1.32, 1.64}
	\draw[stealth-,thick,black]  ($(\x, {sqrt(3)*\x})$) -- ($(\x+ 0.35,{sqrt(3)*\x+ sqrt(3)*0.35})$); 
	\foreach \x in {0, 0.32, 0.64}
	\draw[-stealth,thick,black]  ($(\x, {sqrt(3)*\x})$) -- ($(\x+ 0.35,{sqrt(3)*\x+ sqrt(3)*0.35})$);
	
	\foreach \x in {3, 2.68, 2.36}
	\draw[stealth-,thick,black]  ($(\x, {-sqrt(3)*\x+4*sqrt(3)})$) -- ($(\x- 0.35,{-sqrt(3)*\x+ 4*sqrt(3)+sqrt(3)*0.35})$);  
	\foreach \x in {4, 3.68, 3.36}
	\draw[-stealth,thick,black] ($(\x, {-sqrt(3)*\x+4*sqrt(3)})$) -- ($(\x- 0.35,{-sqrt(3)*\x+ 4*sqrt(3)+sqrt(3)*0.35})$);
	
	\foreach \y in {1.2, 1.7, 2.3, 2.85}
	\draw[stealth-,thick,black]  (2, \y) -- (2,\y+0.5); 
	\foreach \y in {0, .28, .56,.84}
	\draw[-stealth,thick,black]  (2, \y) -- (2,\y+0.3);
	
	\foreach \x in {2.02, 2.32, 2.64}
	\draw[stealth-,thick,black]($(\x, {1/3*sqrt(3)*\x})$) -- ($(\x+ 0.35,{1/3*sqrt(3)*\x+ 1/3*sqrt(3)*0.35})$); 
	\foreach \x in {0, 0.5, 1,1.5}
	\draw[-stealth,thick,black]  ($(\x, {1/3*sqrt(3)*\x})$) -- ($(\x+ 0.45,{1/3*sqrt(3)*\x+ 1/3*sqrt(3)*0.45})$);

	\foreach \x in {2.5,3,3.5,4}
	\draw[-stealth,thick,black] ($(\x, {-1/3*sqrt(3)*\x+4/3*sqrt(3)})$) -- ($(\x- 0.45,{-1/3*sqrt(3)*\x+ 4/3*sqrt(3)+1/3*sqrt(3)*0.45})$);   
	\foreach \x in {1.95,1.65,1.35}
	\draw[stealth-,thick,black]  ($(\x, {-1/3*sqrt(3)*\x+4/3*sqrt(3)})$) -- ($(\x- 0.35,{-1/3*sqrt(3)*\x+ 4/3*sqrt(3)+1/3*sqrt(3)*0.35})$);
	
	\draw (0,0) node[below]{$(0,0,1)$};
	\draw (2,0) node[below]{$(0,\frac{1}{2},\frac{1}{2})$};
	\draw (4,0) node[below]{$(0,1,0)$};
	\draw (3,1.8) node[right]{$(\frac{1}{2},\frac{1}{2},0)$};
	\draw ($(2,{2*sqrt(3)})$) node[above]{$(1,0,0)$};
	\draw (1,1.8) node[left]{$(\frac{1}{2},0,\frac{1}{2})$};
	\draw (2,-1) node[below]{Case $\theta <\frac{2}{3}$};
	
	\path[stealth-,black] (2.1,1.6) edge[bend left=5] node { } (2.15,1.9);
	\path[stealth-,black] (2.5,1.6) edge[bend right=15] node { } (2.7,1.9);
	\path[stealth-,black] (2.5,2) edge[bend right=15] node { } (2.55,2.3);
	\path[stealth-,black] (2.1,2) edge[bend left=5] node { } (2.15,2.3);
	\path[stealth-,black] (2.3,2) edge[bend right=5] node { } (2.35,2.3);
	\path[stealth-,black] (2.35,2.4) edge[bend right=5] node { } (2.3,2.7);
	\path[stealth-,black] (2.15,2.4) edge[bend left=5] node { } (2.17,2.7);
	\path[stealth-,black] (2.15,2.8) edge[bend right=5] node { } (2.1,3.1);
	\path[stealth-,black] (2.2,1.4) edge[bend right=5] node { } (2.35,1.7);
	\path[stealth-,black] (2.25,1.7) edge[bend right=5] node { } (2.42,2);
	
	\path[stealth-,black] (1.82,1.8) edge[bend right=5] node { } (1.85,2.1);
	\path[stealth-,black] (1.82,2.2) edge[bend right=5] node { } (1.85,2.5);
	\path[stealth-,black] (1.6,2.3) edge[bend left=5] node { } (1.7,2.6);
	\path[stealth-,black] (1.75,2.7) edge[bend left=5] node { } (1.83,3);
	\path[stealth-,black] (1.88,2.8) edge[bend right=5] node { } (1.9,3.1);
	\path[stealth-,black] (1.4,2.1) edge[bend left=5] node { } (1.5,2.4);
	\path[stealth-,black] (1.7,1.55) edge[bend left=15] node { } (1.4,1.75);
	\path[stealth-,black] (1.3,1.75) edge[bend left=30] node { } (1.35,2.1);
	\path[stealth-,black] (1.6,1.8) edge[bend left=5] node { } (1.65,2.1);
	\path[stealth-,black]  (1.82,1.4) edge[bend right=5] node { } (1.85,1.7);
	
	\path[stealth-,black] (1.9,1) edge[bend left= 15] node { } (1.8,.7);
	\path[stealth-,black] (1.75,.9) edge[bend left=5] node { } (1.45,.7);
	\path[stealth-,black] (1.4,.65) edge[bend left=5] node { } (1.1,.45);
	\path[stealth-,black] (.9,.35) edge[bend left=5] node { } (.6,.15);
	\path[stealth-,black] (.5,.15) edge[bend left=5] node { } (.2,.03);
	\path[stealth-,black] (1.8,.6) edge[bend left=30] node { } (1.6,.25);
	\path[stealth-,black] (1.5,.25) edge[bend left=10] node { } (1.2,.1);
	\path[stealth-,black] (1.1,.2) edge[bend left=10] node { } (0.8,.1);
	\path[stealth-,black] (.7,.1) edge[bend left=10] node { } (0.4,.035);
	\path[stealth-,black] (1.3,.4) edge[bend left=5] node { } (.95,.25);
	\path[stealth-,black] (1.7,.65) edge[bend left=5] node { } (1.4,.45);
	
	\path[stealth-,black] (.85,.7) edge[bend left=5] node { } (.5,.5);
	\path[stealth-,black] (.75,1.05) edge[bend right=5] node { } (.55,.65);
	\path[stealth-,black] (.4,.45) edge[bend left=1] node { } (.2,.2);
	\path[stealth-,black] (1.7,1.2) edge[bend left=5] node { } (1.4,1);
	\path[stealth-,black] (1.2,1.25) edge[bend right=30] node { } (.9,1.05);
	\path[stealth-,black] (1.2,.85) edge[bend right=5] node { } (.9,.65);
	\path[stealth-,black] (1,1) edge[bend right=5] node { } (.7,0.75);
	\path[stealth-,black] (1.4,1.1) edge[bend right=5] node { } (1.05,0.9);
	\path[stealth-,black] (1.5,1.25) edge[bend right=45] node { } (1.2,1.3);
	\path[stealth-,black] (1.15,1.4) edge[bend right=45] node { } (.9,1.2);
	
	\path[stealth-,black] (2.1,1) edge[bend right=5] node { } (2.15,.7);
	\path[stealth-,black] (2.1,.7) edge[bend right=15] node { } (2.2,.4);
	\path[stealth-,black] (2.2,.3) edge[bend right=45] node { } (2.45,.15);
	\path[stealth-,black] (2.5,.2) edge[bend right=15] node { } (2.8,.15);
	\path[stealth-,black] (3,.1) edge[bend right=5] node { } (3.4,.07);
	\path[stealth-,black] (2.2,.9) edge[bend right=5] node { } (2.5,.65);
	\path[stealth-,black] (2.55,.6) edge[bend right=5] node { } (2.8,.45);
	\path[stealth-,black] (2.95,.4) edge[bend right=5] node { } (3.3,.15);
	\path[stealth-,black] (3.25,.3) edge[bend right=5] node { } (3.7,.05);
	\path[stealth-,black] (2.3,.6) edge[bend right=5] node { } (2.7,.35);
	\path[stealth-,black] (2.7,.25) edge[bend right=5] node { } (3.1,.17);
	
	\path[stealth-,black] (3.5,.5) edge[bend left=5] node { } (3.8,.2);
	\path[stealth-,black] (3.1,.75) edge[bend left=5] node { } (3.4,.55);
	\path[stealth-,black] (2.6,1) edge[bend left=5] node { } (2.9,.8);
	\path[stealth-,black] (2.2,1.15) edge[bend left=5] node { } (2.5,1.0);
	\path[stealth-,black] (2.75,1.2) edge[bend left=5] node { } (3.15,.9);
	\path[stealth-,black] (2.8,1.4) edge[bend left=5] node { } (3.15,1.1);
	\path[stealth-,black] (2.35,1.2) edge[bend left=45] node { } (2.7,1.25);
	\path[stealth-,black] (2.7,1.4) edge[bend left=45] node { } (3.05,1.4);
	\path[stealth-,black] (3.2,1.1) edge[bend left=5] node { } (3.4,.8);
	\path[stealth-,black] (3.2,.85) edge[bend left=5] node { } (3.5,.6);
	\end{tikzpicture}
	\caption{Trajectories of $V_\theta$: the dynamics is different for $\theta > \frac{2}{3}$ and $\theta < \frac{2}{3}$. }\label{f:dyn}
\end{figure}
%
%

Figure~\ref{f:dyn} indicates the different dynamics of $V_\theta$ for $\theta < \frac{2}{3}$ and $\theta>\frac{2}{3}$. See  \cite{JKL17}.
For $\theta < \frac{2}{3}$, $C$ is attracting, $\mathbf{e}_1$ is repelling. Moreover, positive orbits in $\mathrm{int}(G_1)$ converge to $C$. For $\theta>\frac{2}{3}$ the stability of the fixed point is reversed, so $C$ is repelling, $\mathbf{e}_1$ is attracting and
positive orbits in $\mathrm{int}(G_1)$ converge to $\mathbf{e}_1$. The fixed point $c$ is still a saddle fixed point for all $\theta \ne \frac{2}{3}$.

\section{{Main results}}
\label{sec:mainresult}

Let $\mu$ be a probability measure on $[0,1]$ equipped with the Borel $\sigma$-algebra and let $\Theta_n,\,n \in \Z$ be a two-sided sequence of
independent $[0,1]$-valued random variables with law $\mu$.
A sequence of random parameters $\Theta_n$ yields an element $\omega \in \Omega = [0,1]^\mathbb{Z}$.
Denote the left shift operator on $\Omega$ by $\vartheta$: if $\omega = (\Theta_n)_{n\in\mathbb{Z}}$, then  $\vartheta \omega = (\Theta_{n+1})_{n\in\mathbb{Z}}$.
The family $V_{\Theta_n}$, $n \in \Z$ is a family of independent and identically distributed maps taking values in $\Delta^2$ equipped with the
topology inherited by $\R^2$, which defines a
discrete {\em random dynamical system} (RDS) $\varphi_n (x,\omega) := V_{\Theta_n}\circ \cdots \circ  V_{\Theta_1} (x)$ for $n \ge 0$ and
$\varphi_{n} (\omega,x):= V_{\Theta_{n+1}}^{-1} \circ \cdots \circ   V_{\Theta_0}^{-1} (x)$ for $n < 0$.
We also write $\varphi_n (x)$, suppressing the dependence on $\omega$ from the notation.

If $\Theta$ is a random variable with distribution $\mu$,
then we write $\EE g(\Theta)$ for the expected value of $g(\Theta)$ instead of $\int g(y)\,\dd \mu(y)$ in case $g$ is a measurable function for which the integral is defined (possibly $\infty$ or $-\infty$).


We aim to identify the  long term dynamics, in particular the forward point attractor, of the RDS $\varphi$.
A set $A \subset \Delta^2 \times \Omega$
is a
forward point attractor
if
\begin{enumerate}
	\item
$A$  is a compact random set, i.e. $A(\omega) = A \cap (\Delta^2 \times \{\omega\})$ is a nonempty compact set so that $d(x , A(\omega))$ depends measurably on $\omega$ for any $x \in \Delta^2$;
   \item
$A$
is strictly
$\varphi_n$-invariant, i.e.
\[
\varphi_n (A (\omega))  = A (\vartheta^n \omega)
\]
almost surely;
    \item
$A$
attracts points, i.e.
\[
\lim_{n\to \infty}
 d (\varphi_n (x)   , A (\vartheta^n \omega)) = 0
\]
almost surely, for every $x \in \Delta^2$.
\end{enumerate}
A forward point attractor is called minimal if it is the minimal set with these properties. It is shown in  \cite{CS} that minimal point attractors in the weak or pullback sense
always exist. For forward point attractors this question is open but in the particular set-up of the following theorem we will see that a minimal forward  point attractor exists.

The following result describes the minimal forward point attractor  of the RDS $\varphi$.
We will exclude the very special (and uninteresting) case $\mu(\{2/3\})=1$ which means that $\varphi_{n}=\mathrm{Id}$ almost surely for every $n \in \Z$.

\begin{theorem}\label{t:A}
Let $\mu \neq \delta_{2/3}$ be a probability measure on $[0,1]$. Then the minimal forward point attractor of the RDS $\varphi$ is given by $\A$. Moreover,
every trajectory converges almost surely (to one of the points in the set $\A$).
\end{theorem}

\begin{proof}
Since the set of fixed points of $V_\theta$ contains  $\A$ for every $\theta \in [0,1]$ it follows that $\A$ is almost surely contained in any forward (and in any pullback or weak)  point attractor.

To show the converse inclusion, let $x \in \Delta^2$. We need to show that $V_{\Theta_n}\circ \cdots \circ V_{\Theta_1}(x)$ converges to the set $\A$ almost surely.
By symmetry it suffices to consider the RDS on $G_1$.

We start with calculations of the Lyapunov exponents 
at the vertices $\mathbf{e}_1,c,C$ of $G_1$.
Recall that a Lyapunov exponent at a point $x$ is a limit point $\lim_{n\to\infty} \frac{1}{n} \log \| D ( V_{\Theta_n}\circ \cdots \circ V_{\Theta_1}) (x)  v \|$
for a $v \ne 0$, and is used to determine expected contraction or expansion rates of nearby trajectories. See \cite{A98}.

Write $(V_\theta(x))_1 = x_1 W_1 (x)$. Consider the permutation $\sigma (x_1,x_2,x_3) = (x_2,x_3,x_1)$ and let $W_2 =W_1 \circ \sigma$, $W_3 = W_1\circ \sigma^2$.
With this notation,
\[
D V_\theta (x) =
\left(
\begin{array}{ccc}
W_1 (x) + x_1 \frac{\partial}{\partial x_1} W_1 (x) &  x_1  \frac{\partial}{\partial x_2} W_1 (x) & x_1 \frac{\partial}{\partial x_3} W_1 (x)
\\
x_2 \frac{\partial}{\partial x_1} W_2 (x) &  W_2 + x_2  \frac{\partial}{\partial x_2} W_2 (x) & x_1 \frac{\partial}{\partial x_3} W_2(x)
\\
x_3 \frac{\partial}{\partial x_1} W_3 (x) &  x_3  \frac{\partial}{\partial x_2} W_3 (x) & W_3 (x) + x_3 \frac{\partial}{\partial x_3} W_3 (x)
\end{array}
\right).
\]
A calculation shows
\begin{align*}
D V_\theta (C) &=
\left(
\begin{array}{ccc}
\frac{11}{9} + \frac{2}{3}\theta & \frac{8}{9} - \frac{1}{3}\theta & \frac{8}{9} - \frac{1}{3} \theta \\
\frac{8}{9} - \frac{1}{3}\theta     & \frac{11}{9} + \frac{2}{3}\theta &    \frac{8}{9} - \frac{1}{3}\theta \\
\frac{8}{9} - \frac{1}{3}\theta  & \frac{8}{9} - \frac{1}{3}\theta  & \frac{11}{9} + \frac{2}{3}\theta
\end{array}
\right),
\end{align*}
\[
D V_\theta (c) =
\left(
\begin{array}{ccc}
\frac{3}{2} + \frac{3}{4}\theta & \frac{3}{2} - \frac{3}{4}\theta & \frac{1}{2} + \frac{3}{4} \theta \\
\frac{3}{2} - \frac{3}{4}\theta & \frac{3}{2} + \frac{3}{4}\theta & \frac{1}{2} + \frac{3}{4} \theta \\
0 & 0 & 2 - \frac{3}{2} \theta
\end{array}
\right),
\]
\[
D V_\theta (\mathbf{e}_1) =
\left(
\begin{array}{ccc}
3 & 3\theta & 3 \theta \\
0 & 3 (1-\theta) & 0 \\
0 & 0 & 3 (1-\theta)
\end{array}
\right).
\]
By $S_3$-equivariance, $D V_\theta (C)$ acting on $T_C \Delta^2$
is a multiple of the identity.
At $C$, $DV_\theta (C)$ has a multiple eigenvalue; a calculation shows the eigenvalue is $\frac{1}{3} + \theta$.
Likewise  $D V_\theta (\mathbf{e}_1)$ is a multiple of the identity on $T_{\mathbf{e}_1} \Delta^2$ with eigenvalue $3 (1-\theta)$.
By  $S_3$-equivariance, $DV_\theta (c)$ on $T_c \Delta^2$ is diagonalizable.
The two eigenvalues are $\frac{3}{2} \theta$ with eigenvector $(1,-1,0)$ and $2 - \frac{3}{2} \theta$ with eigenvector $(1,1,-2)$.
Note that zero eigenvalues occur for $\theta = 0$ at $c$ and for $\theta=1$ at $\mathbf{e}_1$.


We now prove that the RDS has negative Lyapunov exponents at at least one of the vertices $\mathbf{e}_1$, $C$.
By Birkhoff's ergodic theorem,  the Lyapunov exponents at $\mathbf{e}_1$ and $C$ equal
$\EE \log (3 (1-\Theta))$ and $\EE \log (\frac{1}{3} + \Theta)$ respectively.
%
Using Jensen's inequality, we find that
\begin{align*}
\EE \log (3 (1-\Theta))& \le \log (3 (1-\EE \Theta)),
\\
\EE \log \big(\frac{1}{3} + \Theta\big)&\le \log \big(\frac{1}{3} + \EE \Theta\big).
\end{align*}
If $\EE \Theta \neq 2/3$, then exactly one of the two expressions on the right hand side is (strictly) negative. If $\EE \Theta=2/3$, then both expressions on the right hand side
are zero and both inequalities are in fact strict due to the assumption $\mu \neq \delta_{2/3}$.
Therefore at least one of $\EE \log(3 (1-\Theta))$ and $\EE \log(\frac{1}{3} + \Theta)$ is (strictly) negative.
We conclude that either $\mathbf{e}_1$, or $C$, or both, have negative Lyapunov exponents.




It remains to prove that for $x \in G_1$, $\varphi_n (x)$ converges to one of the equilibria in $\{  \mathbf{e}_1, c, C\}$.
Consider $x \in \mathrm{int}(G_1)$.
Let $\ast \in \{  \mathbf{e}_1, c, C\}$.
Assume that the RDS has negative Lyapunov exponents at $\ast$.
By the local stable manifold theorem \cite[Theorem~(5.1)]{R79}, 
if $\mathcal{U}$ is a small neighborhood of $\ast$ and $x \in \mathrm{int}(G_1) \cap \mathcal{U}$,
we have $\PP (\lim_{n\to\infty} \varphi_n (x) = \ast) \ge \varepsilon$ for some $\varepsilon>0$.
Note that the integrability condition to apply \cite[Theorem~(5.1)]{R79} is satisfied, compare \cite{SV}.
By Lemma~\ref{l:lyapgezero},
if the RDS has nonnegative Lyapunov exponents at $\ast$,
then $\PP (\varphi_n (x) \not \in \mathcal{U} \textrm{ for some }n\in \mathbb{N}) = 1$.
It is further clear that, if $\mathcal{V}$ is the complement in $\mathrm{int}(G_1)$ of a small neighborhood of $\{  \mathbf{e}_1, c, C\}$,
for $x \in \mathcal{V}$ we have $\PP (\varphi_n (x) \not \in \mathcal{V} \textrm{ for some }n\in \mathbb{N}) = 1$.

These properties and the observation that at least one equilibrium in $\{  \mathbf{e}_1, c, C\}$ has negative Lyapunov exponents together easily imply that for
$x \in \mathrm{int}(G_1)$, $\varphi_n (x)$ converges to one of the equilibria in $\{  \mathbf{e}_1, c, C\}$
with probability one.

The same type of reasoning can be used for the RDS restricted to the invariant edges $M_{12}, M_{23}, \Gamma_{12}$ of $G_1$.
Consider for instance $M_{12}$. Then the RDS restricted to $M_{12}$ has a Lyapunov exponent $\EE \log(2 - \frac{3}{2}\Theta)$ at $c$ and a Lyapunov exponent
$\EE\log(\frac{1}{3}+\Theta)$ at $C$. As above, one of these Lyapunov exponents is negative. The above reasoning can now be followed.
\end{proof}

The previous theorem does not provide any information about the stability of the points in the set $\A$. Ideally, we would like to identify, for each $x \in \Delta^2$, the smallest (deterministic)
subset $A (x)$ of  $\A$ for which
\[
\PP \big( \lim_{n \to \infty} d(\varphi_n(x),A(x)) = 0 \big) > 0.
\]
Of course, for each $x \in \A$ we have $A(x)=\{x\}$, but what about other $x \in \Delta^2$?
By symmetry, it suffices to consider $x \in G_1$.
It will turn out that the set $A(x)$ depends on the distribution $\mu$.
We have the following result.

\begin{theorem} \label{t:A(x)}
Assume $\mu \neq \delta_{2/3}$. 
\begin{itemize}
\item[a)]  For $x \in \mathrm{int}(G_1)$, we have
\begin{itemize}
\item[i)] $\mathbf{e}_1 \in A(x)$ iff $\EE \log \big( 3(1-\Theta)\big) <0$;
\item[ii)] $c \in A(x)$ iff both $\EE \log \big( \frac{3}{2} \Theta \big) <0$ and $\EE \log \big( 2 - \frac{3}{2} \Theta \big) <0$;
\item[iii)]  $C \in A(x)$ iff $\EE \log \big( \frac{1}{3} + \Theta \big) <0$.
\end{itemize}
\item[b)] For $x \in M_{12} \setminus \{ C,c\}$, we have
\begin{itemize}
\item[i)] $C \in A(x)$ iff $\EE \log \big( \frac{1}{3} + \Theta \big) <0$;
\item[ii)] $c \in A(x)$ iff $\EE \log \big( 2 - \frac{3}{2} \Theta \big) <0$.
\end{itemize}
\item[c)] For $x \in M_{23}\setminus \{\mathbf{e}_1,C\} $, we have
\begin{itemize}
\item[i)] $\mathbf{e}_1 \in A(x)$ iff $\EE \log \big( 3(1-\Theta)\big) <0$;
\item[ii)] $C \in A(x)$ iff $\EE \log \big( \frac{1}{3} + \Theta \big) <0$.
\end{itemize}
\item[d)] For $x \in \Gamma_{12}\setminus \{ \mathbf{e}_1,c\}$, we have
\begin{itemize}
\item[i)] $\mathbf{e}_1 \in A(x)$ iff $\EE \log \big( 3(1-\Theta)\big) <0$;
\item[ii)] $c \in A(x)$ iff $\EE \log \big( \frac{3}{2} \Theta \big) <0$.
\end{itemize}
\end{itemize}
\end{theorem}

Before proving the theorem, let us consider some special cases.
Take an iterated function system where $\mu$ is supported on two values $\theta - \frac{1}{10}, \theta+\frac{1}{10}$ for a parameter $\theta$:
$\mu = \frac{1}{2} \delta_{\theta - \frac{1}{10}} +  \frac{1}{2}\delta_{\theta+\frac{1}{10}}$.
If $\theta < \frac{2}{3} - \frac{1}{10}$, $C$ is an attracting vertex for both maps $V_{\theta-\frac{1}{10}}$ and $V_{\theta + \frac{1}{10}}$.
If $\theta > \frac{2}{3} + \frac{1}{10}$, $\mathbf{e}_1$ is an attracting vertex for both maps $V_{\theta-\frac{1}{10}}$ and $V_{\theta + \frac{1}{10}}$.
In both cases the dynamics of the iterated function system is clear.
Transitions in the dynamics will take place if $\theta$ runs from $\frac{2}{3}-\frac{1}{10}$ to $\frac{2}{3} + \frac{1}{10}$.
We base our observations on Figure~\ref{f:lyap} that shows graphs of the Lyapunov exponents at the three vertices $c,C,\mathbf{e}_1$.
 From this we see that for varying $\theta$ there are intervals where subsequently $\{ C\}$, $\{ C,\mathbf{e}_1\}$, $\{ C,\mathbf{e}_1,c\}$, $\{c,\mathbf{e}_1\}$ and $\{\mathbf{e}_1\}$
are the vertices with negative maximal Lyapunov exponent.  Boundary points of these intervals are parameter values with zero Lyapunov exponent at a vertex.
In particular there is an open interval of parameter values where all three vertices have negative Lyapunov exponents.
This gives rise to intermingled basins (see \cite{kan94}) for the basins of three attractors: every point in $\mathrm{int(G_1)}$ has positive probability to converge
to each of the three vertices of $G_1$.
Compare \cite{GH} for the context of iterated function systems on an interval and an explanation of the terminology.
Remarkably, the random system can converge to $c$ with positive probability for every starting point $x\in \mathrm{int}(G_1)$, while
this convergence is not possible for the deterministic system.
\begin{figure}[!ht]

\begin{center}
\includegraphics[width=12cm]{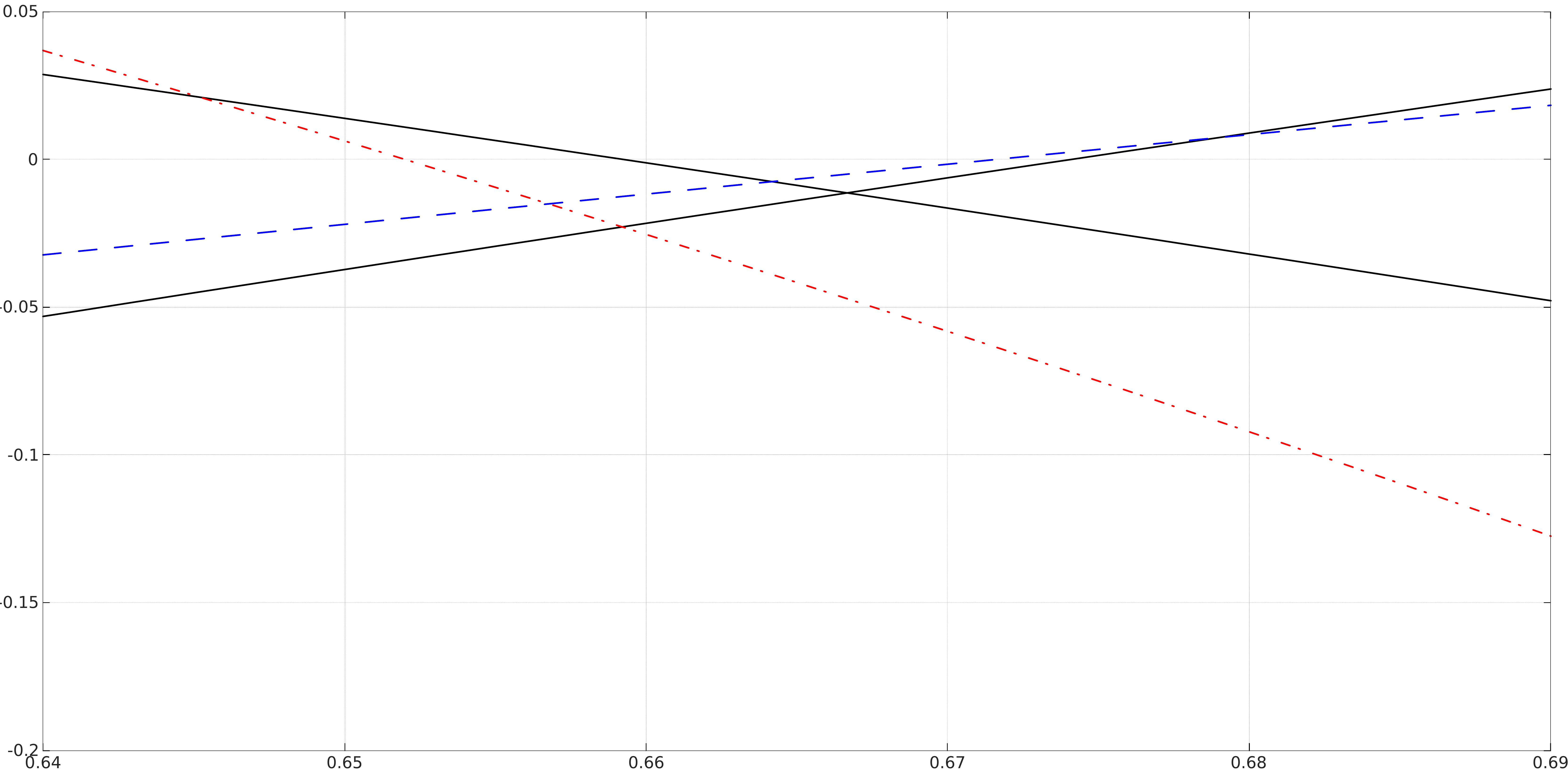}
\caption{Graphs of Lyapunov exponents at the three vertices $c,C,\mathbf{e}_1$. The two solid curves depict two Lyapunov exponents of the random saddle $c$,
the increasing dashed  curve depicts the Lyapunov exponent of $C$ (the two exponents are equal)
the remaining decreasing dash-dotted curve depicts the Lyapunov exponent of $\mathbf{e}_1$ (the two exponents are equal).
\label{f:lyap}}
\end{center}

\end{figure}

Assume that $\mu$ is Lebesgue measure, so that $\Theta$ is chosen uniformly from $[0,1]$.
A calculation gives the following identities:
\begin{align*}
\EE \log (3(1-\Theta)) &= \log(3) - 1 > 0,
\\
\EE \log(\frac{3}{2} \Theta) &= \log(\frac{3}{2})-1 < 0,
\\
\EE \log(2-\frac{3}{2}\Theta) &= \frac{5}{3}\log(2)-1>0,
\\
\EE \log(\frac{1}{3} + \Theta) &= \log (\frac{4}{3} 2^{2/3})-1 < 0.
\end{align*}
It follows that for any $x \in \textrm{int}(G_1)$, $\varphi_n (x)$ converges to the center point $C$ almost surely, as $n\to \infty$.


Assume that $\Theta$ is uniformly distributed on $[\frac{2}{3}-a,\frac{2}{3}+a]$, $0<a\le \frac{1}{3}$, so that $\EE \Theta = \frac{2}{3}$.
Since $\log(x)$ is a concave function,  $\EE \log (3(1-\Theta))$, $\EE \log(\frac{3}{2} \Theta)$,
$\EE \log(2-\frac{3}{2}\Theta)$ and
$\EE \log(\frac{1}{3} + \Theta)$ are all negative.
The situation is similar to the example of the iterated function system.
For any $x \in \mathrm{int}(G_1)$, trajectories $\varphi_n (x)$ converge almost surely
and each point $\mathbf{e}_1,c,C$ occurs as limit point with positive probability.

\begin{proof}[Proof of Theorem~\ref{t:A(x)}]
We will only prove item~a). The proof of item~a) gives all the ingredients for the other items, which can then easily be concluded.

Let us first consider the vertex $\mathbf{e}_1$.
Recall from the proof of Theorem~\ref{t:A} that at $\mathbf{e}_1$ the  Lyapunov exponents equal $\EE \log (3 (1-\Theta))$.
So the condition $\EE \log (3 (1-\Theta)) <0$ means that the Lyapunov exponents at $\mathbf{e}_1$ are negative.
The following statement was already observed in the proof of Theorem~\ref{t:A} and is a consequence of the local stable manifold theorem.
For $\varepsilon >0$ small and $x \in B_\varepsilon (\mathbf{e}_1) = \{ x \in G_1 \; ; \; d (x,\mathbf{e}_1) < \varepsilon \}$,
\[
\PP \big( \lim_{n \to \infty} d(\varphi_n(x), \mathbf{e}_1) = 0 \big) > 0.
\]
That is, $\mathbf{e}_1 \in A(x)$ for $x \in  B_\varepsilon (\mathbf{e}_1)$.
It now follows from Lemma~\ref{l:Asaddle}  that $\mathbf{e}_1 \in A(x)$ for all $x \in \mathrm{int}(G_1)$.
Lemma~\ref{l:lyapgezero} demonstrates that $\mathbf{e}_1 \not\in A(x)$ in case $\EE \log (3 (1-\Theta)) \ge 0$.
We conclude that for $x \in \mathrm{int} (G_1)$, $\mathbf{e}_1 \in A(x)$ iff $\EE \log \big( 3(1-\Theta)\big) <0$.
The same reasoning can be followed for the edge $C$, so that for any $x \in \mathrm{int}(G_1)$ we have that $C \in A(x)$  precisely if
$\EE \log (\frac{1}{3} + \Theta) < 0$.

The random saddle point $c$ is more difficult to treat.
By earlier calculations, the Lyapunov exponents at $c$ are $\EE \log \big( \frac{3}{2} \Theta \big)$ and $\EE \log \big( 2 - \frac{3}{2} \Theta \big)$.
Under the condition of negative Lyapunov exponents at $c$, the support of $\mu$ has nonempty intersection with both $[0,\frac{2}{3})$ and $(\frac{2}{3},1]$.
Lemma~\ref{l:Asaddle} provides, for given $x \in \mathrm{int}(G_1)$ and $\varepsilon>0$, an integer $N\in \mathbb{N}$ so that
\[
\PP \big( d(\varphi_N(x), c) < \varepsilon \big) > 0.
\]
As before we conclude that for any $x \in \mathrm{int}(G_1)$, we find $c \in A(x)$
under the assumption that both Lyapunov exponents at $c$ are negative.
If one of the Lyapunov exponents is nonnegative, $c \not\in A(x)$ by Lemma~\ref{l:lyapgezero}.
\end{proof}

The following lemma is specific to the family $V_\theta$ and an ingredient for the proof of Theorem~\ref{t:A(x)}.

\begin{lemma}\label{l:Asaddle}
Assume the support of $\mu$ intersects $[0,\frac{2}{3})$.	
	Then for any $\varepsilon>0$ and $x \in \mathrm{int} (G_1)$, there exists $N \in \mathbb{N}$ with
	\[
	\PP \big( d(\varphi_N(x), C) < \varepsilon \big) > 0.
	\]
	
	Assume the support of $\mu$ intersects $(\frac{2}{3},1]$.	
	Then for any $\varepsilon>0$ and $x \in \mathrm{int} (G_1)$, there exists $N \in \mathbb{N}$ with
	\[
	\PP \big( d(\varphi_N(x), \mathbf{e}_1) < \varepsilon \big) > 0.
	\]
	
Assume the support of $\mu$ intersects both $[0,\frac{2}{3})$ and $(\frac{2}{3},1]$.
Then for any $\varepsilon>0$ and $x \in \mathrm{int} (G_1)$, there exists $N \in \mathbb{N}$ with
\[
 	\PP \big( d(\varphi_N(x), c) < \varepsilon \big) > 0.
\]
\end{lemma}

\begin{proof}
	Consider first the fixed point $\mathbf{e}_1$.
	It follows from \cite{JKL17} that for $x \in \mathrm{int}(G_1)$ there exists $\theta$ in the support of $\mu$ and $N\in\mathbb{N}$
	so that $V_\theta^N (x) \in B_\varepsilon(\mathbf{e}_1)$. By continuity of $V_\theta$ in $\theta$, we find that for given $x \in \mathrm{int}(G_1)$ there is $N\in\mathbb{N}$ so that
	$
	\PP \big( d(\varphi_N(x), \mathbf{e}_1) < \varepsilon \big) > 0.
	$
	The statement on $C$ goes similarly.
	
	To prove the statement on $c$,
	we start with an analysis of random iterates near $C$.
	Put $x_3 = 1 - x_1 - x_2$ in $V_\theta$ and consider the resulting map
	\begin{align*}
	F_\theta(x_1,x_2) &=
	\\
	\MoveEqLeft
	\left(
	\begin{array}{c}
	x_1 \left( x_1^2 + 3 \theta x_1 (1-x_1) + 3(1-\theta) (x_2^2 + (1-x_1-x_2)^2) + 2 x_2(1-x_1-x_2)\right) \\ x_2 \left(x_2^2 + 3 \theta x_2(1-x_2) + 3(1-\theta) (x_1^2 + (1-x_1-x_2)^2) + 2 x_1 (1-x_1-x_2)  \right)
	\end{array}
	\right).
	\end{align*}


	Compute $DF_\theta (\frac{1}{3},\frac{1}{3}) =  (\frac{1}{3} + \theta) \left( \begin{array}{cc}  1 & 0 \\ 0 & 1 \end{array}\right) $.
	The Hessian of the first component $F_{\theta,1}$ of $F$  equals
	\[
	\frac{1}{3} (4- 6 \theta) \left( \begin{array}{cc}  -1 & 1 \\ 1 & 2 \end{array}\right).
	\]
	Likewise the Hessian of $F_{\theta,2}$ equals
	\[
	\frac{1}{3} (4- 6 \theta) \left( \begin{array}{cc}  2 & 1 \\ 1 & -1 \end{array}\right).
	\]
	So the second order Taylor expansion of $F_\theta$ around $(\frac{1}{3},\frac{1}{3})$ equals, in terms of $x = x_1-\frac{1}{3}, y = x_2 - \frac{1}{3}$,
	\[
	T_\theta (x,y) =  (\frac{1}{3} + \theta) \left( \begin{array}{c} x \\ y \end{array} \right) + \frac{1}{3} (2 - 3 \theta) \left( \begin{array}{c} - x^2 +  2xy + 2y^2 \\  2x^2 + 2xy - y^2 \end{array} \right).
	\]
	Denote $\lambda_\theta =  \frac{1}{3} + \theta$.
	Write $T_\theta (x,y)  = \lambda_\theta \left( \begin{array}{c} x \\ y \end{array} \right) + P_\theta (x,y)$.
	Consider a local coordinate transformation $H_\theta (x,y) =  \left( \begin{array}{c} x \\ y \end{array} \right) + U_\theta (x,y)$
	with
	$U_\theta (x,y) = \frac{-1}{\lambda_\theta - \lambda_\theta^2} P_\theta (x,y)$.\\


	Calculate
	\begin{align*}
	H_\theta^{-1} \circ F_\nu \circ H_\theta  (x,y) &=  H_\theta^{-1} \left[  \lambda_\nu H_\theta (x,y) + P_\nu \circ H_\theta(x,y) \right]
	\\
	&=  H_\theta^{-1} \left[  \lambda_\nu \left( \begin{array}{c} x \\ y \end{array} \right) + \lambda_\nu U_\theta (x,y) + P_\nu (x,y) + \mathcal{O}(3) \right]
	\\
	&= \lambda_\nu \left( \begin{array}{c} x \\ y \end{array} \right) +  \lambda_\nu   U_\theta (x,y) - U_\theta (\lambda_\nu (x,y))
	+ P_\nu (x,y) + \mathcal{O}(3)
	\\
	&= \lambda_\nu \left( \begin{array}{c} x \\ y \end{array} \right) + (\lambda_\nu - \lambda_\nu^2) U_\theta (x,y) + P_\nu (x,y) +\mathcal{O}(3)
	\\
	&= \lambda_\nu \left( \begin{array}{c} x \\ y \end{array} \right) + P_\nu (x,y) - \frac{\lambda_\nu - \lambda_\nu^2}{\lambda_\theta - \lambda_\theta^2} P_\theta (x,y)+\mathcal{O}(3).
	\end{align*}
	Here $\mathcal{O}(3)$ stands for terms of at least third order in $(x,y)$.
	Consequently  $H_\theta^{-1} \circ F_\theta \circ H_\theta$ is linear up to quadratic order.
	We note that there is in fact a coordinate change, equal to $H_\theta$ up to quadratic order, that smoothly linearizes $F_\theta$.
	The formulas show the effect of the same coordinate change $H_\theta$ on  $F_\nu$.
	
	Consider lines $y = kx$ with $-\frac{1}{2} \le k \le 1$.
	The line with $k=-\frac{1}{2}$ contains $M_{23}$,
	the line with $k=1$ contains $M_{12}$.
	So the collection of lines covers $G_1$.
	We will calculate $Q_\nu = H_\theta^{-1} \circ F_\nu \circ H_\theta$ in the point $(x,kx)$.  Note first
	\begin{align*}
	Q_\nu (x,y) &=   (\frac{1}{3} + \nu) \left( \begin{array}{c} x \\ y \end{array} \right) + \\
	& \qquad  \left[ \frac{1}{3} (2-3\nu) - \frac{  (\frac{1}{3} + \nu) -  (\frac{1}{3} + \nu)^2}{ (\frac{1}{3} + \theta) - (\frac{1}{3} + \theta)^2} \frac{1}{3} (2-3\theta) \right]
	\left( \begin{array}{c} -x^2 + 2 xy + 2y^2 \\ 2 x^2 + 2xy - y^2 \end{array} \right),
	\end{align*}
	ignoring terms of order three.
	The expression on the right hand side can be simplified to
	\begin{align*}
	Q_\nu (x,y) &=  (\frac{1}{3} + \nu) \left( \begin{array}{c} x \\ y \end{array} \right)
	+ \frac{1}{3} (2 - 3 \nu) \left[ \frac{\theta - \nu}{\frac{1}{3} + \theta} \right] \left( \begin{array}{c} -x^2 + 2 xy + 2y^2 \\ 2 x^2 + 2xy - y^2 \end{array} \right).
	\end{align*}
	So
	\begin{align*}
	Q_\nu (x,kx) &= (\frac{1}{3} + \nu) \left( \begin{array}{c} x \\ kx \end{array} \right)
	+ \frac{1}{3} (2 - 3 \nu) \left[ \frac{\theta - \nu}{\frac{1}{3} + \theta} \right]  x^2 \left( \begin{array}{c} -1 + 2 k + 2k^2 \\ 2 +2k - k^2 \end{array} \right).
	\end{align*}
	It follows that a point $(x,kx)$ is mapped to a point $(u, v)$ with
	\[
	v/u = k - x \left(  \frac{1}{3} (2-3\nu) \frac{\theta-\nu}{(\frac{1}{3} + \nu) (\frac{1}{3} + \theta)} \right)   (k-1) (2k+1) (k+2) + \mathcal{O}(2) (k-1)(2k+1),
	\]
	where $\mathcal{O}(2)$ here are terms of at least quadratic order in $x$.
	
	Assume $\frac{2}{3} < \theta < 1$ (so that $C$ is unstable for $V_\theta$) and $\nu < \frac{2}{3}$ ($C$ is stable for $V_\nu$).
	This is possible by the assumption on the support of $\mu$.
	We then find that (for $x>0$ small and $-\frac{1}{2}<k<1$) $v/u > k$.
	
	Take a fundamental domain $D = B_\delta(C) \setminus V^{-1}_\theta (B_\delta (C))$ for $V_\theta$. Here $\delta$ is a small positive number,
	so that $D$ is close to $\mathbf{e}_1$.
	Take any $p \in \textrm{int} (G_1)$.  A number of iterates $V_\nu^N$ maps $p$ to a point close to $C$, which we may assume to be in $D$
	by composing with iterates of $V_\theta$ if needed.
	For $q \in D$ and $N > 0$, let $M \in \mathbb{N}$ be so that $V_\theta^M \circ V_\nu^N (q) \in D$. If $N$ is large enough we know $M>0$.
	By the previous calculations, if $q$ lies on the line $y = kx$ then $V_\theta^M \circ V_\nu^N (q)$ lies on a line $y = \ell x$ for $\ell > k$.
	Several such compositions map a point in $D$ to a point still in $D$ but now close to $M_{12}$
	(namely on a line $y = \ell x$ for an $\ell$ close to $1$).
	Finally, there is an iterate $V_\theta^O$ that maps the last point to a point close to $c$.
	By continuity of the maps $V_\theta$ in $\theta$, the same holds true if we take nearby parameter values from small balls
	around $\nu$ and $\theta$.
\end{proof}	

The next lemma discusses consequences of positive or vanishing Lyapunov exponents at the fixed points.
The analysis of dynamics near a fixed point with vanishing Lyapunov exponents is the more delicate case, compare also
the study of random interval diffeomorphisms in \cite{GH}.

\begin{lemma}\label{l:lyapgezero}
	Assume $\EE \log (3 (1-\Theta)) \ge 0$.
	Let $x \in \mathrm{int}(G_1)$.
	For $\varepsilon>0$ small,
	\[
	\PP ( d(\varphi_n(x) ,\mathbf{e}_1) >  \varepsilon  \textrm{ for some } n \in \mathbb{N}) = 1.
	\]
	
	Assume $\EE \log (\frac{1}{3} + \Theta) \ge 0$.
	Let $x \in \mathrm{int}(G_1)$.
	For $\varepsilon>0$ small,
	\[
	\PP ( d(\varphi_n(x) , C) >  \varepsilon  \textrm{ for some } n \in \mathbb{N}) = 1.
	\]
	
	Assume  $\EE \log (\frac{3}{2} \Theta) \ge 0$  or $\EE \log (2 - \frac{3}{2} \Theta) \ge 0$.
	Let $x \in \mathrm{int}(G_1)$.
	For $\varepsilon>0$ small,
	\[
	\PP ( d(\varphi_n(x) , c) >  \varepsilon  \textrm{ for some } n \in \mathbb{N}) = 1.
	\]
\end{lemma}

\begin{proof}
	We first prove the statement on $\mathbf{e}_1$.
	Let $r (x) $ be the Euclidean distance of $x$ to $\mathbf{e}_1$.
	Define $s:\mathrm{int}(G_1)\times [0,1] \to (0, \infty]$ such that
	\begin{align}\label{e:logr}
	\log r (V_\theta (x)) &= \log r (x) + \log (3 (1-\theta)) +  \log s (x,\theta).
	\end{align}
	Then, denoting $\delta:=1-x_1$ and $\gamma:=1-\theta$, we get
	$$
	s^2(x,\theta)=\frac 1{9\gamma^2 r^2(x)}\Big(\Big(1-\big(V_\theta (x)\big)_1\Big)^2+\Big(\big(V_\theta (x)\big)_2\Big)^2+ \Big(\big(V_\theta (x)\big)_3\Big)^2\Big),
	$$
	where
	\begin{align*}
	1-\big(V_\theta (x)\big)_1&=f_1(x)+\gamma g_1(x),\\
	\big(V_\theta (x)\big)_2&=f_2(x)+\gamma g_2(x),\\
	\big(V_\theta (x)\big)_3&=f_3(x)+\gamma g_3(x),
	\end{align*}
	with
	\begin{align*}
	f_1(x)&=3\delta^2 - 2\delta^3 - 2 (1-\delta)x_2x_3, \\
	g_1(x)&=3(1-\delta)(\delta-r^2(x)), \\
	f_2(x)&=x_2^3+3x_2^2(x_3+1-\delta)+2x_3(1-\delta)x_2,\\
	g_2(x)&=3x_2\big(x_3^2+(1-\delta)^2-x_2(x_3+1-\delta)\big),\\
	f_3(x)&=x_3^3+3x_3^2(x_2+1-\delta)+2x_3(1-\delta)x_2, \\
	g_3(x)&=3x_3\big(x_2^2 +(1-\delta)^2 -x_3(x_2+1-\delta)\big).
	\end{align*}
	Note that $f_i(x)\ge 0$ for all $i \in \{1,2,3\}$ and $x \in \mathrm{int}(G_1)$ (consider the case $\gamma=0$).
	Also, $g_i (x) \ge 0$ for all $i \in \{1,2,3\}$, at least for $r(x)$ small.
	Therefore,
	\begin{align*}
	s^2(x,\theta)&\ge s^2(x,0)=\frac 1{9r^2(x)}\sum_{i\in \{1,2,3\}} (f_i(x)+g_i(x))^2 \\
	&\ge\frac 1{9r^2(x)} \Big(9\delta^2 + 9 x_2^2 + 9x_3^2-O(r^4(x))\Big)=1-O(r^2(x)).
	\end{align*}
	
	
	We can therefore define $t(w) = -Ce^{w}$ with $C>0$ so that $t(\log r(x)) \le \log s(x,\theta)$ for all $\theta$ and $r(x) < \varepsilon$.
	Here $A_0 = \log (\varepsilon)$ is assumed to be a large negative number.
	Consider the random walk on the line
	\begin{align}\label{e:rww}
	w_{n+1} = w_n + \log(3(1-\Theta_{n+1})) + t (w_n).
	\end{align}
	We compare orbits of \eqref{e:rww} with $\varphi_n (x)$.
	Whenever $w_i < \log (\varepsilon)$, $0 \le i \le n$ and $w_{n+1} > \log (\varepsilon)$,
	we have $r(\varphi_{n+1} (x)) >\varepsilon$.
	So we must show that there exists some $A_0 \in \R$ such that for every deterministic initial condition $w_0 < A_0$, we have
	\[
	\PP(w_n \ge A_0 \textrm{ for some } n \in \mathbb{N}) = 1.
	\]
	This is standard if $\EE \log(3(1-\Theta)) >0$. We assume now the case of vanishing Lyapunov exponents $\EE \log(3(1-\Theta)) = 0$.
	
	Fix $A_0 \le -1$ and consider the stopped walk $W_{n}:=w_{n\wedge \tau}$, where $\tau:= \inf\{m \in \N:\,w_m\ge A_0\}$.
	We define the {\em Lyapunov function} $V(x)=\log (-x)$, $x \le -1$. We claim that
	\begin{equation}\label{supermart}
	\EE  \big(V(W_{n+1})|\F_n\big) \le V(W_n) \mbox{ a.s.~on the set } \{W_n<A_0\},
	\end{equation}
	where $\F_n$ denotes the $\sigma$-algebra generated by $W_1,\,W_2,\cdots,W_n$.
	This means that the process $n \mapsto W_n$ is a nonnegative supermartingale which therefore converges almost surely to
	an integrable random variable  $W_\infty$ which immediately implies that $\tau<\infty$ almost surely since $\log(3(1-\Theta))$ is not deterministic.\\
	

	Let us prove \eqref{supermart}. All we have to show is that
	$$
	\EE \log \Big(1-\frac {\Gamma +t(w)}{-w} \Big) \le 0
	$$
	for all sufficiently small $w\le -1$ where $\Gamma=\log(3(1-\Theta))$. In fact all we will use about $\Gamma$ is that is has mean 0 and  is not
	almost surely equal to 0 and bounded from above. Note that due to the upper boundedness of $\Gamma$ the term inside the $\log$ is almost surely positive
	whenever $w$ is sufficiently small.
	
	We decompose $\Gamma$ as $\Gamma=\Gamma_1+\Gamma_2$, where $\Gamma_1:=\Gamma \eins_{U}$ and $\Gamma_2:=\Gamma \eins_{U^c}$ and where the set $U$ is chosen such
	that $\EE \Gamma_1=\EE \Gamma_2=0$ and $\Gamma_1$ is bounded and not identically 0 (it may be necessary to enlarge the probability space for such a
	decomposition to exist).
	
	Now,
	\begin{equation}\label{terms}
	\EE \log \Big(1-\frac {\Gamma +t(w)}{-w} \Big)=\EE \log\Big(1-\frac {\Gamma_1 +t(w)}{-w} \Big)+ \EE \log \Big( 1-\frac{\Gamma_2}{-w-\Gamma_1-t(w)}\Big).
	\end{equation}
	The second term can be estimated as follows:
	$$
	\EE \log \Big( 1-\frac{-\Gamma_2}{-w-\Gamma_1-t(w)}\Big)\le \EE \frac{\Gamma_2}{-w-\Gamma_1-t(w)}=-\frac 1{-w-t(w)}\EE \big(\Gamma \eins_U\big)=0.
	$$
	It remains to estimate the first term on the right hand side of \eqref{terms}.
	
	For $\varepsilon \in (0,1)$ there exists $\delta >0$ such that
	$$
	\log(1+x)\le x -\frac 12 (1-\varepsilon)x^2, \mbox{ for all } |x|\le \delta.
	$$
	Therefore, for $w\le -1$ sufficiently small we have
	\begin{align*}
	\EE \log \Big(1-\frac {\Gamma_1 +t(w)}{-w} \Big)&\le  -  \frac 1{-w}\big(\EE \Gamma_1 + t(w)\big)
	- \frac 12 (1-\varepsilon) \EE \Big(\frac {\Gamma_1 +t(w)}{-w}\Big)^2\\
	&= \frac {-t(w)}{-w}-\frac 1{2w^2} (1-\varepsilon)\big(\EE \Gamma_1^2+t(w)^2\big)
	\end{align*}
	which is smaller than 0 for $w$ sufficiently small since $\EE \Gamma_1^2>0$ and $\lim_{w \to -\infty}w\,t(w)=0$. This finishes the proof of the first statement in the lemma.
	
	The corresponding statement concerning $C$ is proved similarly (it is in fact easier as the eigenvalues of $D V_\theta (C)$ are bounded away from zero uniformly in $\theta$) and will not be included here.
	
	Finally we prove the statement for $c$.
	Take coordinates $u_1 = u_1(x) = x_1-x_2$, $u_2 = u_2(x) = 1 - x_1-x_2$ for which $u_1 (c) = u_2 (c) = 0$.
	Note that $u_1 \ge 0$ and $u_2 \ge 0$ on $G_1$.
	Writing
	$\gamma_1 = \frac 32 \theta$, $\gamma_2 = 2 - \frac 32 \theta$,
	a calculation shows
	\begin{align*}
	u_1 ( V_\theta (x) ) &=
	\gamma_1 u_1 (1 - u_1^2 - 3 u_2^2 ) + u_1 (u_1^2 + 3 u_2^2),
	\\
	u_2 ( V_\theta (x) ) &=
	\gamma_2 u_2 ( 1 - 4 u_2 + u_1^2 + 3  u_2^2) + u_2 ( 4 u_2 - u_1^2 - 3 u_2^2).
	\end{align*}
	Let $s_1 (x), s_2 (x)$ be defined by
	\begin{align*}
	\log (u_1 ( V_\theta (x) )) &= \log (u_1(x))  + \log (\gamma_1) + \log (s_1 (x)),
	\\
	\log (u_2 ( V_\theta (x) )) &= \log (u_2(x))  + \log (\gamma_2) + \log (s_2 (x)).
	\end{align*}
	We find $s_1 (x) \ge 1 - O ( \| (u_1,u_2) \|^2)$, $s_2 (x) = 1 - O ( \| (u_1,u_2) \|^2)$.

	
	Assume for definiteness $\EE \log(\gamma_1) = 0$.  Using Jensen's inequality (see the proof of Theorem~\ref{t:A}) yields  $\EE \log(\gamma_2) <0$.
	The case $\EE \log(\gamma_1) > 0$ and  $\EE \log(\gamma_2) <0$ is treated similarly and also the case  $\EE \log(\gamma_2) \ge 0$ and $\EE \log(\gamma_1) < 0$ is treated similarly.
	
	For constants $k>0$, $L<0$,
	consider the random walk
	\begin{align*}
	w_{n+1} &= (w_n + \log (\gamma_2) + k e^{L} + k e^{w_n}) \wedge 0
	\end{align*}
	We may pick $k>0$ and $L$ large negative so that $\log (u_2 (\varphi_n (x))) \le w_n$ as long as $w_n \le L$.
	In fact, this holds uniformly in $u_1 (\varphi_n(x)) \le L$.
	There is no loss of generality in assuming $L=0$ and $k$ is small positive.
	
	We first consider the values $w_n$.
	For $k$ small enough we have
	$\EE \gamma_2 - 2 k <0$. This implies that $e^{w_n}$ converges to zero almost surely.
	To be specific,
	write $\omega = (\Theta_n)_{n\in\mathbb{N}}$.
	There are $\lambda<1$ and a measurable function $C_\omega$ so that
	for $\mu^\mathbb{N}$ almost all $\omega$, $e^{w_n} \to 0 $ as $n\to \infty$ and
	\[
	e^{w_n} \le C_\omega \lambda^n e^{w_0}.
	\]
	
	Now consider
	\begin{align*}
	v_{n+1} &= (v_n + \log (\gamma_1) - k e^{v_n} - k e^{w_n}) \wedge 0. 
	\end{align*}
	As long as $w_n\le 0$ and $v_n \le 0$, we find $v_n \le \log (u_1(\varphi_n(x)))$.
	We use the exponential convergence of $e^{w_n}$ to zero to treat the iterates $v_n$.
	Assume first $e^{w_0}=0$ and thus $e^{w_n} = 0$ for all $n\in\mathbb{N}$.
	This gives the random walk
	\begin{align} \label{e:vn0}
	v^0_{n+1} &= (v^0_n + \log (\gamma_1) - k e^{v^0_n}) \wedge 0.
	\end{align}
	As in the previous reasoning on iterates near $\mathbf{e}_1$ the following statement is derived:
	there is a set of $\omega$'s of full measure, for which for any given $v_0 \le 0$, there are infinitely many $n$ with $v_n=0$.
	
	Now consider the random walk $v_{n+1} = (v_n + \log (\gamma_1) - k e^{v_n} - k e^{w_n} ) \wedge 0$
	with the term involving $e^{w_n}$ included.
	We obtain
	\[
	|v_m - v^{0}_m| \le \sum_{i=0}^{m-1} e^{w_i} \le C_\omega \frac{\lambda}{\lambda-1} e^{w_0}.
	\]
	Let $v^{0,n}_k$ be the orbit for \eqref{e:vn0} with $v^{0,n}_0 = v_n$.
	Then
	\begin{align}\label{e:vv0}
	|v_{n+k} - v^{0,n}_k| &\le \sum_{i=0}^{k-1} e^{w_{n+k}}  \le C_{\sigma^n \omega} \frac{\lambda}{\lambda-1} e^{w_n}.
	\end{align}
	Let $B>0$ be such that $C_\omega < B$ for positive probability.
	For almost all $\omega$,  there are arbitrary large $n$ with $C_{\sigma^n \omega}<B$.
	The right hand side of \eqref{e:vv0} goes to zero if $n\to \infty$ and $C_{\sigma^n \omega}<B$.
	We already know that with probability one there are infinitely many $k$ with $v^{0,n}_k=0$.
	This suffices to conclude the result.
\end{proof}

\section*{Acknowledgments}
The second author (UJ)  thanks the TU Berlin  for the kind hospitality and for  providing all facilities and the  German Academic Exchange Service (DAAD) for providing financial support by a scholarship.

\end{document}